\documentclass[12pt]{article}

\usepackage[a4paper,twoside,
 inner=0.72in, outer=0.72in,
 top=1in, bottom=1in
]{geometry}

\usepackage{xcolor}
\usepackage{hyperref}

\hypersetup{
 colorlinks=true,
 linkcolor=blue,
 citecolor=blue,
 urlcolor=blue
}
 
\usepackage{graphicx}  
\usepackage{amsthm, amsfonts, amsmath, amssymb, color, bm}

\numberwithin{equation}{section}
 
\newtheorem{theorem}{Theorem}[section]
\newtheorem{lemma}[theorem]{Lemma}
\newtheorem{corollary}[theorem]{Corollary}
\newtheorem{proposition}[theorem]{Proposition}
\newtheorem{definition}[theorem]{Definition}
\newtheorem{remark}[theorem]{Remark}
\newtheorem{example}[theorem]{Example}

\def\<{\langle}
\def\>{\rangle}

\title{Einstein connection \\ of nonsymmetric pseudo-Riemannian manifold}

\author{
 Vladimir Rovenski\footnote{
Department of Mathematics,
Faculty of Natural Science, University of Haifa,
Mount Carmel, 3498838 Haifa, Israel.
E-mail: {vrovenski@univ.haifa.ac.il}
}
\ and
 Milan Zlatanovi\'c\footnote{
Department of Mathematics, Faculty of Science and Mathematics, University of Ni\v s, Vi\v segradska 33, 18000 Ni\v s, Serbia.
E-mail: {zlatmilan@yahoo.com}
}
}

\date{}

\begin{document}

\maketitle

\abstract{
A.~Einstein considered a linear connection $\nabla$ with torsion $T$ on a smooth manifold 
equipped with a nonsymmetric (0,2)-tensor $G=g+F$, where $g$ is a pseudo-Riemannian metric associated with gravity, and $F\ne0$ is a skew-symmetric tensor associated with electromagnetism, such that $(\nabla_X\,G)(Y,Z)=-G(T(X,Y),Z)$. 
In this paper, we  explicitly present the Einstein connection of a nonsymmetric pseudo-Riemannian manifold with non-degenerate $F$, 
satis\-fying the $f^2$-torsion condition
$T(f^2X,Y)=T(X,f^2Y)=f^2 T(X,Y)$, where $g(X,fY)=F(X,Y)$,
and show that in the almost Hermitian case, it
reduces to the 
M.~Prvanovi\'c's (1995) solution.
We~also discuss special Einstein connections, 
explicitly present the Einstein connection of almost contact metric manifolds satisfying the $f^2$-torsion condition,
and 
give example in terms of
weighted product of almost Hermitian~manifolds.


}


\section{Introduction}

A. Einstein worked on various versions of the Unified Field Theory (Nonsymmetric Gravitational Theory -- NGT), see~\cite{Ein}. This theory was intended to unite the gravitation theory, to which General Relativity is related, and the theory of electromagnetism. Beginning in 1950, A. Einstein used a real nonsymmetric basic (0,2)-tensor $G$,
whose symmetric part $g$ is associated with gravity, and the skew-symmetric part $F$ is associated with electromagnetism.
In~his attempt to construct a Unified Field Theory,
(briefly NGT), A.~Einstein \cite{Ein} considered a smooth manifold $(M,G=g+F)$ with a linear connection $\nabla$ with torsion $T(X,Y)=\nabla_XY-\nabla_YX-[X,Y]$ satisfying the following condition:
\begin{equation}\label{metein0}
 (\nabla_X G)(Y,Z)=-G(T(X,Y),Z)\quad (X,Y,Z\in\mathfrak{X}_M).
\end{equation} 
called an {Einstein connection}.
The~idea of a nonsymmetric basic tensor was revisited by J.W.~Moffat \cite{Moffat-95}, allowing for richer gravi\-tational dynamics. Results by T. Jansen and T.~Prokopec \cite{JP-2007} point to potential improvements and challenges within NGT.
M.I.~Wanas et al. \cite{Wanas-24} studied path equations of a particle in Einstein's nonsymmetric geo\-metry. Recent approaches to modified gravity often rely on differen\-tial geometry, including torsion and non-metricity, as a natural extension of General Relativity.
Connections with {totally skew-symmetric torsion},
i.e., $T(X,Y,Z)$ is a~3-form, satisfying
$T(X,Y,Z)=-T(X,Z,Y)=-T(Y,X,Z)$,
are important due to the relations with mathematical physics, as they admit a coupling to spinor fields and lead to holonomy and rigidity phenomena. 
In this case, the dimension of the space of parallel
spinors coincides with the number of preserved supersymmetries. Such connections arise
naturally in supersymmetric string theories and non-linear $\sigma$-models, as well as in gravitational models \cite{Ham}.
S.~Ivanov and M.\,Lj.~Zlatanovi\'c (see~\cite{IZ1}) presented conditions for the existence and uniqueness of the Einstein connection with totally skew-symmetric torsion on a manifold $(M,G=g+F)$ and gave its explicit expression using an almost contact metric structure.
In~this case, the torsion is determined by
$T(X,Y,Z)=-\frac13\,dF(X,Y,Z)$,
and the diffe\-rence tensor $K:=\nabla-\nabla^g$ is given by 
\begin{equation}\label{newnbl-K}
 K_XY=\tfrac12\big\{T({f}X,Y)-T(X,{f}Y)+T(X,Y)\big\}.
\end{equation}
Totally skew-symmetry of $K$ implies totally skew-symmetry of $T$, but the converse does not hold in general. In this paper, we do not assume that the torsion is totally skew-symmetric. 

An important class of linear connections,
having applications in theory of probability and statistics as well as in information geometry,
are statistical connections, which are
torsion-free and the tensor 
$({\nabla} g)(X,Y,Z):=({\nabla}_X\,g)(Y,Z)$ is symmetric with respect to all entries, equiva\-lently, the cubic form $K(X,Y,Z):=g(K(X,Y),Z)$
(associated with the difference tensor (\ref{eq:contorsion-def}) of $\nabla$) is symmetric.
Unfortunately, we cannot directly apply the statistical structure to Einstein's nonsymmetric geometry: if the cubic form $K(X,Y,Z)$ is symmetric, then by \eqref{E-tordfnew} and~\eqref{E-tordfnew2}, $T=K=0$ is true; hence, $\nabla=\nabla^g$.

\smallskip

Weak metric structures, see \cite{rov-survey24}, gene\-ralize the almost Hermitian and almost contact metric structures, as well as K.~Yano's $f$-structure, and are well suited for studying $(M,G=g+F)$ in NGT with an arbitrary skew-symmetric tensor~$F$.
A \textit{weak almost Hermitian manifold} $(M,f,g)$ is a Riemanni\-an manifold $(M,g)$ of even dimension $\ge4$ endowed with non-singular skew-symmetric (1,1)-tensor $f$ and the 2-form 
$F(X,Y)=g(X,fY)$, see~\cite{rst-63}.
V.~Rovenski and M.~Zlatanovi\'c~\cite{rst-63,RZ-1} were the first to apply the weak almost Hermitian and 
weak almost contact structures to NGT of totally skew-symmetric torsion with $F(X,Y) = g(X,fY)$.
M.~Prvanovi\'c \cite{Prvanovic-95} obtained the explicit form of the Einstein connection of almost Hermitian manifold,
and using its torsion tensor, characterized the sixteen classes (by A.~Gray and L.M. Hervella \cite{gh-1980}) of almost Hermitian manifold. 

\smallskip

In this paper, we  explicitly present the Einstein connection of an NGT space with non-degenerate $F$, that is, a weak almost Hermitian manifold, satisfying some natural condition, called the $f^2$-torsion condition. This condition disappears for the classical case, hence our result generalizes the result by M.~Prvanovi\'c. 
We also  explicitly present the Einstein connection (and the special Einstein connection
defined by the property $K_XY=-K_YX$) of an almost contact metric manifold satisfying the $f^2$-torsion~condition. 

The paper has an Introduction and five sections.
In~Section~\ref{sec:geom} we review basics of NGT, give examples, prove several auxiliary statements, and discuss the relation between the torsion and contorsion of an Einstein connection emphasizing the role of $f^2$-torsion condition.
In~Section~\ref{sec:a.c.m} we find the Einstein connection of an almost contact metric mani\-fold with the $f^2$-torsion condition 
(Theorems~\ref{thm72} and \ref{thm72-s}).
In~Section~\ref{sec:w-Herm} we 
present the Einstein connection of nonsymmetric pseudo-Riemannian manifold with non-degenerate $F$ and the $f^2$-torsion condition (Theorem~\ref{T-8.2}), in particular, for weak almost Hermitian manifold (Corollary~\ref{C-5.2}). 
Long expressions (terms $\delta_i$) of Section~\ref{sec:w-Herm}
are given in Section~\ref{sec:app}.

\section{Einstein's nonsymmetric geometry}
\label{sec:geom}

A manifold $(M,G=g+F)$ equipped with an Einstein connection
$\nabla$ different from the Levi-Civita connection $\nabla^g$ is called an NGT space.
The basic (0,2)-tensor $G$ of an NGT space 
$(M,G=g+F)$ decomposes into two parts, the symmetric part $g$ (pseudo-Riemannian metric) and the skew-symmetric part $F$ (fundamental 2-form), that is,
\begin{equation*}
 g(X,Y)=\tfrac12\big[G(X,Y)+G(Y,X)\big], \quad F(X,Y)=\tfrac12\big[G(X,Y)-G(Y,X)\big].
\end{equation*}
We assume that the skew-symmetric part, $F$ is non-zero, may have arbitrary (not necessarily constant) rank, and in particular, may be non-degenerate (that is, of maximal rank).
Therefore, we obtain a (1,1)-tensor ${f}\ne0$ (of not necessarily constant rank) determined by
\begin{equation*}
 g(X, fY) = F(X,Y)\quad (X, Y \in \mathfrak{X}_M).
\end{equation*}
Since $F$ is skew-symmetric, the tensor ${f}$ is also skew-symmetric:
 $g({f}X, Y) = -g(X,{f}Y)$.
Note that $f=-A$, where the tensor $A$ is given in \cite{IZ1} by the equality $g(AX, Y) = F(X,Y)$.


The torsion of a linear connection $\nabla$ on $M$ is given by 
\[
 T(X,Y)=\nabla_XY-\nabla_YX-[X,Y].
\]
The covariant derivative an $(p,k)$-tensor $S$ for $p=0,1$ is an $(p,k+1)$-tensor $\nabla S$:
\begin{equation}\label{E-nablaP-S0}
 (\nabla S)(Y,X_{1},\ldots,X_{k}) =
 \nabla_{Y}(S(X_{1},\ldots,X_{k}))
 -\sum\nolimits_{\,i=1}^{k}S(X_{1},\ldots,
 \nabla_{Y}X_{i},\ldots,X_{k}) .
\end{equation}
Separating symmetric and skew-symmetric parts of \eqref{metein0}, 
we express the covariant derivatives $\nabla g$ and $\nabla F$ in terms of the (0,3)-torsion tensor $T(X,Y,Z)=g(T(X,Y),Z)$:
\begin{align}\label{ein6}
 2\,(\nabla_X\,g)(Y,Z)
&= T(Z,X,\,Y+fY) -T(X,Y,\,Z+fZ) ,\\
\label{ein5}
 2\,(\nabla_Z\,F)(X,Y)
&= - T(Z,X,\,Y+fY) - T(Y,Z,\,X+fX) .
\end{align} 
Recall the co-boundary formula for a 2-form $F$ 
(without the coefficient 3, unlike \cite{Blair-survey}):
\begin{align}\label{E-3.3a}
 dF(X,Y,Z) &= X(F(Y,Z)) + Y(F(Z,X)) + Z(F(X,Y)) \notag\\
 &\ -F([X,Y],Z) - F([Z,X],Y) - F([Y,Z],X).
\end{align}
The equality \eqref{E-3.3a} yields
\begin{align}\label{E-3.3}
 dF(X,Y,Z) =
 (\nabla^g_X\,F)(Y,Z)+(\nabla^g_Y\,F)(Z,X)+(\nabla^g_Z\,F)(X,Y).
\end{align}
The metricity condition \eqref{metein0} can be written  in the form
\begin{equation}\label{metricityeqq}
 (\nabla_X(g+F))(Y,Z) = -\,T(X,Y,Z)-T(X,Y,fZ).
\end{equation} 
Taking the cyclic sum in \eqref{metricityeqq} and applying the equality
\begin{equation*}
\begin{aligned}
dF(X,Y,Z)
&=
\,T(X,Y,fZ)
+\,T(Y,Z,fX)
+\,T(Z,X,fY)\\
&+(\nabla_X F)(Y,Z)
+(\nabla_Y F)(Z,X)
+(\nabla_Z F)(X,Y) ,
\end{aligned} 
\end{equation*} 
which follows from \eqref{ein5} and \eqref{E-3.3}, we obtain
\begin{equation}\label{ein2
3.2}
\begin{aligned}
&\quad (\nabla_X g)(Y,Z) +(\nabla_Y g)(Z,X) +(\nabla_Z g)(X,Y)  \\
&= -\,dF(X,Y,Z) -\,T(X,Y,Z) -\,T(Y,Z,X) -\,T(Z,X,Y).
\end{aligned}
\end{equation} Since the left hand side of (\ref{ein2
3.2}) is symmetric, while the right hand side is
skew-symmetric, we get the following two relations, see~\cite[Eq.~(3.3)]{IZ1}:
\begin{align*}
&(\nabla_X\,g)(Y,Z)
+(\nabla_Y\,g)(Z,X)
+(\nabla_Z\,g)(X,Y)
= 0, 
\end{align*}
and
\begin{align}\label{ein8}
 dF(X,Y,Z) = -T(X,Y,Z) -T(Y,Z,X) -T(Z,X,Y).
\end{align} 

The {\it difference tensor} $K$ of a linear connection $\nabla$ and the Levi-Civita connection $\nabla^g$ is
\begin{equation}\label{eq:contorsion-def}
 K_X Y := \nabla_X Y - \nabla^g_X Y.
\end{equation}
The Einstein connection $\nabla$
is represented in \cite[Eq.~(3.7)]{IZ1} using the torsion $T$ as
\begin{align}\label{genconein}
 g(\nabla_XY,Z) &=g(\nabla^g_XY,Z) 
 +\tfrac12\big\{T(X,Y,Z) -T(Z,X,{f}Y) +T(Y,Z,{f}X)\big\}.
\end{align}
By \eqref{genconein}, the difference (0,3)-tensor 
$K(X,Y,Z)=g(K_X Y, Z)$ and the torsion tensor of an Einstein connection $\nabla$ are expressed linearly in terms of each~other:
\begin{align}\label{E-tordfnew}
 2\,K(X,Y,Z)&= T(X,Y,Z) -T(Z,X,{f}Y)+T(Y,Z,{f}X),\\
\label{E-tordfnew2}
 T(X,Y,Z) &= K(X,Y,Z) - K(Y,X,Z) \quad\Longleftrightarrow\quad 
 T(X,Y) = K_X Y - K_Y X .
\end{align}

Using various presentations of tensors and \eqref{E-nablaP-S0}, we obtain
\begin{align}\label{important}
 (\nabla^g F)(Z,X,Y)= (\nabla^g_{Z}\,F)(X,Y)
 = g(X, (\nabla^g_Z\,f)Y).    
\end{align}

\begin{lemma}\label{L-nabla-g}
For an Einstein connection $\nabla$, the following conditions are equivalent: 

(i) $K(X,Y,Z)=-K(X,Z,Y)$, 

(ii) $\nabla g=0$,

(iii) $(\nabla_X F)(Y,Z)=-(\nabla_Y F)(X,Z)$,
or $g((\nabla_X f)Z, Y) = -g((\nabla_Y f)Z, X)$.
\end{lemma}

\begin{proof}
 We calculate using \eqref{E-nablaP-S0}  and \eqref{eq:contorsion-def},
\begin{align*}
 (\nabla_X\,g)(Y,Z) & 
 = X(g(Y,Z)) - g(\nabla^g_X Y+ K_XY, Z) - g(Y, \nabla^g_X Z+ K_XZ) \\
 & = (\nabla^g_X\,g)(Y,Z) - g(K_XY, Z) - g(K_XZ, Y).
\end{align*}
By this and $\nabla^g\,g=0$, the conditions (i) and (ii)
are equivalent.
The equations \eqref{ein6} and \eqref{ein5} yield, see  \cite[Eq.~(3.4)]{IZ1},
\begin{align}\label{ein7}
 (\nabla_Z\,g)(X, Y) = (\nabla_X F)(Y,Z) +(\nabla_Y F)(X,Z).
\end{align}
By \eqref{ein7} and \eqref{important},
the conditions (ii) and (iii) are equivalent.
\end{proof}

\begin{remark}\rm
When $\nabla$ preserves the metric tensor: $\nabla g=0$,
they call $K$ the \textit{contorsion tensor}; 
in this case the tensors $K$ and $T$ are related by
\begin{align}\label{E-tordfnew3}
 2\,K(Y,Z,X) = T(X,Y,Z) + T(Y,Z,X) - T(Z,X,Y).
\end{align}
Comparing \eqref{E-tordfnew3} and \eqref{E-tordfnew} yields the
following identity for the torsion of an Einstein connection
satisfying $\nabla g=0$:
\begin{align}\label{E-T4a}
 T(Z,X,Y) - T(Y,Z,X) = T(Y,Z,{f}X) -T(Z,X,{f}Y) .
\end{align}
\end{remark}

\begin{remark}\rm
Since for an Einstein connection on  $(M,G=g+F)$, the tensor $K(X,Y,Z)$ is not totally skew-symmetric, it is interesting to study its particular symmetries. 
The~skew-symmetry of $K_X$ i.e., 
$K(X,Y,Z)=-K(X,Z,Y)$, see Lemma~\ref{L-nabla-g}(i),
by \eqref{E-tordfnew}, reduces~to
\begin{align}\label{E-cond-E2}
 T(X,Y,Z) -T(Z,X,Y) - T(Z,X,fY)+T(X,Y,fZ)=0 .
\end{align}
Using \eqref{E-cond-E2} in \eqref{ein5}, we obtain 
\begin{align*}
 (\nabla_X F)(Y,Z)= -T(X,Y,Z) -T(X,Y,fZ).
\end{align*}
If \eqref{E-cond-E2} is true, then $[K_X,K_Y]: TM\to TM$ is a skew-symmetric endomorphism:
\begin{align}\label{E-cond-KKZ}
 [K_X, K_Y](Z,Z)=0,\quad X,Y,Z\in\mathfrak{X}_M.
\end{align}
If a nonsymmetric Riemannian manifold $(M,G=g+F)$ admits an Einstein connection $\nabla$ with totally skew-symmetric torsion, and 
\eqref{E-f-torsion} holds, then \eqref{E-cond-KKZ} is true.
Indeed, using \eqref{E-f-torsion} and the totally skew-symmetric torsion property in the LHS of \eqref{E-cond-E2} gives~zero.
\end{remark}

The following property $K_XY=-K_YX$ of the difference tensor $K$ 
(i.e., the skew--symmetry of \(K(X,Y,Z)\) only in the first two arguments) by \eqref{E-tordfnew} reduces~to
\begin{align}\label{E-cond-2K}
 T(Y,Z,fX) + T(X,Z,fY) = 0,
\end{align}
and characterizes \emph{special Einstein connections},
i.e., the symmetric part of an Einstein connection $\nabla$ coincides with the Levi-Civita connection $\nabla^g$, see~\cite{Prvanovic-95}.
In this case, \eqref{E-tordfnew2} yields $K_XY=\tfrac12\,T(X,Y)$.
The~condition \eqref{E-cond-2K} doesn't imply the total skew--symmetry of \(K(X,Y,Z)\), which, in addition, requires the condition \eqref{E-cond-E2}. 

\begin{lemma}\label{L-KT-1} 
If the tensor $K$ is totally skew-symmetric, then the torsion tensor $T$ is totally skew-symmetric, and the $f$-torsion condition \eqref{E-f-torsion} is true.
\end{lemma}

\begin{proof}
Let the tensor $K(X,Y,Z)$ be totally skew-symmetric. Then
we obtain 
\[
K(X,Y,Z) + K(Y,X,Z) = 0 \implies T(Y,Z,fX) -T(Z,X,fY) = 0.
\]
By Lemma~\ref{L-nabla-g}, we have $\nabla g=0$.
Combining this identity with \eqref{E-T4a} and using the assumption that $K(X,Y,Z)$ is totally skew-symmetric, we deduce
\[
T(X,Y,Z) = -T(X,Z,Y)\quad \text{and}\quad T(X,Y,fZ) = -T(X,Z,fY).
\]
So, $T(X,Y,Z)$ is totally skew-symmetric and the $f$-torsion condition \eqref{E-f-torsion} is true.
\end{proof}

\begin{proposition}
Let the condition
\eqref{E-cond-E2} be true. 
Then we get the equality 
\begin{align}\label{E-df-0b}
 K(X,Y,Z) = T(Z,Y,X) - \tfrac12\,dF(X,Y,Z);  
\end{align}
therefore, the tensor $K$ is totally skew-symmetric if and only if the torsion $T$ is totally skew-symmetric. Moreover, the total skew-symmetry of $T$ implies the 
condition \eqref{E-f-torsion}. 
\end{proposition}

\begin{proof}
Applying \eqref{E-tordfnew2} to \eqref{ein8} gives
\begin{align}\label{E-df-1}
\notag
 & K(X,Y,Z) - K(Y,X,Z) 
 +K(Y,Z,X) - K(Z,Y,X) \\
 +& K(Z,X,Y) - K(X,Z,Y) = - dF(X,Y,Z).   
 \end{align}
Then applying \eqref{E-cond-E2} to \eqref{E-df-1}, yields 
\begin{align}\label{E-df-2}
\notag
  & K(X,Y,Z) - K(Y,X,Z) 
 -K(Y,X,Z) +K(Z,X,Y) 
 +K(Z,X,Y) +K(X,Y,Z) \\
  & = 2\big(K(X,Y,Z) - K(Y,X,Z) +K(Z,X,Y) \big)
  = - dF(X,Y,Z).   
 \end{align}
 Using \eqref{E-tordfnew2} to \eqref{E-df-2} gives \eqref{E-df-0b}.
 By the above,  the total skew-symmetry of $T$ implies the total skew-symmetry of~$K$.
 By Lemma~\ref{L-KT-1}, the $f$-torsion condition 
\eqref{E-f-torsion} is true.
\end{proof}


\begin{proposition}[see Proposition~2 of \cite{Prvanovic-95}]
For an Einstein connection $\nabla$ with torsion $T$, we have 
\begin{align}\label{2.7^*}
\notag
 2(\nabla^g_X F)(Y,Z)= & -T(Z,X,Y) -T(X,Y,Z) -T(fZ,X,fY)  \\
 & -T(X,fY,fZ) +T(Y,fZ,fX) +T(fY,Z,fX). 
\end{align} 
\end{proposition}

\begin{proof} 
By \eqref{E-nablaP-S0},
we have
\begin{align*}
(\nabla_X F)(Y,Z)
=
X\big(F(Y,Z)\big)
- F(\nabla_X Y, Z)
- F(Y,\nabla_X Z).
\end{align*}
Using the relation \eqref{eq:contorsion-def},
we obtain
\begin{align}\label{E1-2.22}
\notag
(\nabla_X F)(Y,Z)
\notag
&= X\big(F(Y,Z)\big)
- F(\nabla^g_X Y + K_XY, Z)
- F(Y, \nabla^g_X Z + K_XZ) \\
\notag
&=(\nabla^g_X F)(Y,Z)- F\big(K_XY,Z\big) -F\big(Y,K_XZ\big)\\
\notag
&=(\nabla^g_X F)(Y,Z)- g\big(K_XY,fZ\big) +g\big(fY,K_XZ\big)\\
&=(\nabla^g_X F)(Y,Z)- K(X,Y,fZ) +K(X,Z,fY) ,
\end{align}
Finally, \eqref{2.7^*} follows from \eqref{ein5}, \eqref{E-tordfnew} and \eqref{E1-2.22}.
\end{proof} 


\section{Einstein connection of an almost contact metric manifold}
\label{sec:a.c.m}

Contact Riemannian geometry is a significant field in both pure mathe\-matics and theoretical physics,
for example, \cite{Blair-survey}. 
Developing the approach of \cite{Prvanovic-95}, 
we deduce equations for 
Einstein connection of an almost contact metric manifold $(M^{2n+1},f,\xi,\eta,g)$.
Here, $g$ is a Riemannian metric, 
$f$ is a skew-symmetric (1,1)-tensor of rank $2\,n$, $\xi$ is a unit (Reeb) vector field, and $\eta$ is a 1-form such that $\eta(\xi) = 1$. These tensors satisfy the~relations 
\begin{align}\label{eq-7.0a}
 f^2 = -{I} + \eta\otimes\xi, \quad
 g(f X, f Y) = g(X,Y) - \eta(X)\eta(Y) ,
\end{align}
where ${I}$ is the identity map. 
Differentiating by $X$ the first identity of \eqref{eq-7.0a}, yields
\[
\nabla^g_X f^2 = \nabla^g_X(-{I}+\eta\otimes\xi)
=\nabla^g_X(\eta\otimes\xi).
\]
Using the Leibniz rule, we have
\begin{align*}
 \nabla^g_X(\eta\otimes\xi) = (\nabla^g_X\eta)\otimes\xi
 +\eta\otimes \nabla^g_X\xi, \qquad
\nabla^g_X f^2=(\nabla^g_Xf)\,f +f\,\nabla^g_Xf .
\end{align*}
Applying the equalities above to $Y$, yields
\begin{equation}\label{eq31}
(\nabla^{g}_{X}f)\,fY+f\,(\nabla^{g}_{X}f)Y
=(\nabla^{g}_{X}\eta)(Y)\,\xi+\eta(Y)\,\nabla^{g}_{X}\xi.
\end{equation}

\begin{definition}\rm
A linear connection $\nabla$ satisfies the $f^2$-\textit{torsion condition} if the following is~true:
\begin{align}\label{E-Q-torsion}
 T(f^2X, Y) = T(X,f^2Y) = f^2 T(X,Y) ,
\end{align} 
which can be equivalently written as
\begin{align*}
 T(f^2X, Y,Z) = T(X,f^2Y,Z) = T(X,Y,f^2Z) .
\end{align*} 
\end{definition}

If \eqref{E-Q-torsion} is true for an Einstein connection, then, using \eqref{ein8} and \eqref{2.7^*}, we have
\begin{align*}
 & dF(f^2X, Y, Z) = dF(X,f^2Y, Z) = dF(X,Y,f^2Z), \\
\notag
 & (\nabla^g F)(f^2X, Y, Z) = (\nabla^g F)(X,f^2Y, Z) 
  = (\nabla^g F)(X,Y,f^2Z).
\end{align*} 
If the $f$-torsion condition \eqref{E-f-torsion} holds, then the $f^2$-torsion condition \eqref{E-Q-torsion} also holds. 
The~converse is not true in general.

\begin{lemma}
The following 
conditions are true for all vector fields
$X,Y,Z$ on an almost contact metric manifold:
\begin{align}\label{eq32b}
(\nabla^g_X F)(f Y,f Z) &= -(\nabla^g_X F)(Y,Z)
+ \eta(Y)\,(\nabla^g_X F)(\xi,Z) + \eta(Z)\,(\nabla^g_X F)(Y,\xi) , \\
\label{eq32c}
(\nabla^g_X F)(f Y,Z) &= (\nabla^g_X F)(Y,f Z)
+\eta(Z)\,(\nabla^g_X F)(f Y,\xi) -\eta(Y)\,(\nabla^g_X F)(\xi,f Z).
\end{align}
\end{lemma}

\begin{proof}
Using $\nabla^g g=0$ and $F(Y,Z)=g(Y,f Z)$, we have
$(\nabla^g_X F)(Y,Z)=g\big(Y,(\nabla^g_Xf)Z\big)$.
Hence
\[
(\nabla^g_X F)(f Y,f Z) =g\big(f Y,(\nabla^g_Xf)\,f Z\big).
\]
Substituting \eqref{eq31}, we get
\[
(\nabla^g_Xf)\,f Z
=-f\,(\nabla^g_Xf)Z +(\nabla^g_X\eta)(Z)\xi +\eta(Z)\nabla^g_X\xi.
\]
Using $g(\xi,f Z)=0$ gives
\begin{align*}
(\nabla^g_X F)(f Y,f Z)
&=-\,g\left(f Y, f\,(\nabla^g_Xf)Z\right)
+\eta(Z)\,g(f Y, \nabla^g_X\xi).
\end{align*}
By $g(f X,f Y)=g(X,Y)-\eta(X)\eta(Y)$ and $\eta(f Z)=0$, we obtain
\[
g\!\left(f Y, f(\nabla^g_Xf)Z\right)
=g\big(Y,(\nabla^g_Xf)Z\big)-\eta(Y)\eta\big((\nabla^g_Xf)Z\big).
\]
Therefore,
\begin{align}\label{eq3.5}
(\nabla^g_X F)(f Y,f Z) =-g\big(Y,\nabla^g_Xf)Z\big) +\eta(Y)\eta\big((\nabla^g_Xf)Z\big) 
+\eta(Z)\,g(fY, \nabla^g_X\xi).
\end{align}
Finally, substituting the equalities
\begin{align*}
g\big(Y,(\nabla^g_Xf)Z\big)&=(\nabla^g_XF)(Y,Z),\quad
\eta\big((\nabla^g_Xf)Z\big)=g\big(\xi,(\nabla^g_Xf)Z\big)=(\nabla^g_XF)(\xi,Z), \\
(\nabla^g_XF)(Y,\xi)&=g\big(Y,(\nabla^g_Xf)\xi\big)
=-g\big(Y,f\nabla^g_X\xi\big)=g(\nabla^g_X\xi,f Y) .
\end{align*}
in \eqref{eq3.5}, yields \eqref{eq32b}. 
Taking $Z\to f Z$ in \eqref{eq32b}, we get \eqref{eq32c}.
\end{proof}


\begin{theorem}\label{thm72}
Let $(M^{2n+1},f,\xi,\eta,g)$ be an almost contact metric manifold, considered as an NGT space $(M,G=g+F,\nabla)$ with $F(X,Y)=g(X,fY)$.
If the Einstein connection $\nabla$ satisfies the $f^2$-torsion condition \eqref{E-Q-torsion}, then
\begin{align}\label{eq:TXiZX}
 \nabla^g\,\xi =\nabla^g\,\eta= 0, 
\quad 
 T(\xi,\cdot,\cdot)=T(\cdot,\cdot,\xi)=0,
\quad 
 (\nabla^g F)(\xi,\cdot,\cdot)=(\nabla^gF)(\cdot,\cdot,\xi)=0.
\end{align}
In particular, the torsion $T$ is horizontal, and for $X,Y,Z\perp\xi$ it is given by the same formula \eqref{Eq-2.13} as for the almost Hermitian case:
\begin{align}\label{Eq-phi-T0}
\notag
 2\,T(Y,Z,X) =&\ 2\,(\nabla^g_{fX} F)(fY,Z)
  - (\nabla^g_{fY} F)(fZ,X) -(\nabla^g_{fZ} F)(X,fY) \\
 &-(\nabla^g_Y F)(Z,X) -(\nabla^g_Z F)(X,Y) .
\end{align}
\end{theorem}

\begin{proof}
We write $Z=f^2Z'+\eta(Z)\xi$, where $Z'\perp\xi$, and by 
\eqref{E-Q-torsion}, we~get
\begin{align}\label{eq77b}
 T(\xi,Z,\xi)=T(\xi,f^2Z',\xi) +\eta(Z)\,T(\xi,\xi,\xi)
 =T(\xi,f^2Z',\xi)=T(f^2\xi,Z',\xi)=0.
\end{align} 
Using 
the $f^2$-torsion condition \eqref{E-Q-torsion} again
and \eqref{eq77b}, we obtain
\begin{align*}
 T(X,Z,\xi)&=T(X,f^2Z',\xi) +\eta(Z)\,T(X,\xi,\xi)
 =T(f^2X,Z',\xi)=0, \\
 T(\xi,Z,X)&=T(\xi,f^2Z',X)+\eta(Z)\,T(\xi,\xi,X)
 =T(f^2\xi,Z',X)=0.
\end{align*} 
Hence,  
\begin{align}\label{eq77c}
T(\xi,Z,\xi)=T(X,Z,\xi)=T(\xi,Z,X)=0.
\end{align}
Substituting $X=\xi$ into \eqref{2.7^*}, we obtain
\begin{align}\label{Eq-acm0}
 \nabla^g_\xi\,F =0,
 \quad\mbox{that is,}\quad \nabla^g_\xi\,f=0.
\end{align}
Therefore, the $\xi$-trajectories are $g$-geodesics.
From \eqref{2.7^*},
we have
\begin{align}\label{Eq-acm1a}
 (\nabla^g_X F)(Y,Z)-(\nabla^g_X F)(fY,fZ) = T(Y,fZ,fX)+T(fY,Z,fX) .
\end{align}
Taking into account (\ref{eq32b}), (\ref{eq77c}) and \eqref{Eq-acm0}, from \eqref{Eq-acm1a} we obtain
\begin{align}\label{Eq-acm1}
2 (\nabla^g_X F)(Y,Z)= T(Y,fZ,fX)+T(fY,Z,fX) .
\end{align}
Replacing $X\to f X$ and $Y\to f Y$ in \eqref{Eq-acm1}, we get
\begin{align*}
2 (\nabla^g_{fX} F)(fY,Z)=- T(fY,fZ,X)+T(Y,Z,X) 
=-2 (\nabla^g_{fX} F)(fZ,Y).
\end{align*}
{
We calculate, following the approach of \cite{Prvanovic-95}:
\begin{align}\label{Eq-2.4}
(\nabla^g_X F)(fY,fZ)=-(\nabla^g_X F)(Y,Z), \qquad
(\nabla^g_X F)(fY,Z)=(\nabla^g_X F)(Y,fZ).
\end{align}
From \eqref{metein0} with $G=g+F$ we obtain
\begin{align}
\notag
2\,(\nabla^g_X F)(Y,Z) =&\,
-T(Z,X,Y)-T(X,Y,Z) -T(X,fY,fZ) \\
\label{Eq-2.7}
& -T(fZ,X,fY) + T(Y,fZ,fX)+T(fY,Z,fX).
\end{align}
Using \eqref{Eq-2.7}, we get
\[
(\nabla^g_X F)(Y,Z)
-
(\nabla^g_X F)(fY,fZ)
=
T(Y,fZ,fX)+T(fY,Z,fX).
\]
Substituting this into \eqref{Eq-2.7} and using \eqref{Eq-2.4}, we obtain
\begin{align*}
 T(X,Y,Z)+T(Z,X,Y) + T(X,fY,fZ)+T(fZ,X,fY) = 0.
\end{align*}
Thus \eqref{Eq-2.7} reduces to
\begin{align}\label{Eq-2.9}
 2\,(\nabla^g_X F)(Y,Z) = T(Y,fZ,fX)+T(fY,Z,fX).
\end{align}
The cyclic sum of \eqref{Eq-2.9} is
\begin{align}\label{Eq-2.10}
\notag
(\nabla^g_X F)(Y,Z)+(\nabla^g_Y F)(Z,X)+(\nabla^g_Z F)(X,Y) \\
= -T(X,Y,Z)-T(Y,Z,X)-T(Z,X,Y).
\end{align}
Replacing $Y\to fY$ and $Z\to fZ$ in \eqref{Eq-2.10} gives
\begin{align}
\label{Eq-2.11}
\notag
(\nabla^g_X F)(fY,fZ)
+(\nabla^g_{fY} F)(fZ,X)
+(\nabla^g_{fZ} F)(X,fY) \\
= T(X,Y,Z) +T(Z,X,Y) -T(fY,fZ,X).
\end{align}
Replacing $X\to fX$ and $Y\to fY$ in 
\eqref{Eq-2.9}, we obtain
\begin{align}\label{Eq-2.12}
T(fY,fZ,X)= T(Y,Z,X) - 2\,(\nabla^g_{fX} F)(fY,Z).
\end{align}
Substituting \eqref{Eq-2.12} into \eqref{Eq-2.11}, gives the formula
\eqref{Eq-phi-T0}.
}
\end{proof}

\begin{remark}\rm
Although the $f^2$-torsion condition becomes trivial for almost Hermitian mani\-folds, it is strong enough for almost contact metric manifolds.
In Theorem~\ref{thm72}, our manifold is locally the direct product of $\mathbb R$
and an almost Hermitian manifold $(\bar M^{2n}, J, \bar g)$, and $\nabla$, being an Einstein connection on the $2n$-dimensional factor, is given by \eqref{Eq-2.13}, where $J$ is replaced by $f$.
Consequently, the 1-form $\eta$ is closed, $d\eta = 0$. 
The distribution $\ker\eta$ is involutive; thus, it is Frobenius integrable. The integral submanifold $\bar M^{2n}$ is almost~Hermitian.
\end{remark}

We work with an Einstein connection  satisfying the $f^2$-torsion condition, without assuming any symmetry of the torsion tensor.
If, in addition, the torsion $T$ is totally skew--symmetric, then the general formula \eqref{Eq-phi-T0} reduces to \eqref{Eq-phi-T0-skew}, and $\nabla^g_\xi F=0$.
This additional assumption leads to the almost--nearly cosymplectic case, which we recall below from \cite{IZ1}.

\begin{corollary}
Under the conditions of Theorem~\ref{thm72}, the torsion $T$ is totally skew-symmet\-ric if and only if
\(
\)
\eqref{eq:TXiZX} 
is true,
and the following Codazzi-type condition holds:
\begin{equation}\label{codazzi}
(\nabla^g_X F)(Y,Z)
=(\nabla^g_Y F)(Z,X)
=(\nabla^g_Z F)(X,Y)
=\tfrac13\, dF(X,Y,Z).
\end{equation}
In this case, the torsion 
is given by
\begin{equation}\label{Eq-phi-T0-skew}
T(X,Y,Z)=-(\nabla^g_Z F)(X,Y)
=-\tfrac13\, dF(X,Y,Z).
\end{equation}
\end{corollary}

\begin{proof}
By Theorem~\ref{thm72}, we have 
$\nabla^g_\xi\,F =0$ and $\nabla^g\,\xi=0$.
Let's consider the $f^2$-torsion condition under the additional assumption that $T$ is totally skew-symmetric. 
Applying $T(Y,X,Z)=-T(Y,Z,X)$, from (\ref{Eq-phi-T0}) we have
\begin{equation}\label{3fcond}
 0 = 3(\nabla^g_{fX} F)(fY,Z) +3(\nabla^g_{fZ} F)(fY,X)
   -(\nabla^g_Z F)(X,Y) -(\nabla^g_X F)(Z,Y) .
\end{equation}
Taking $X\to fX$ and $Y\to fY$ in \eqref{3fcond} and using 
\eqref{eq:TXiZX} and the equality $(\nabla^g_X F)(fY,fZ)=-(\nabla^g_X F)(Y,Z)$, see \eqref{eq32b}, we~get
\begin{align*}
\notag
 0=\ 3\,(\nabla^g_{X} F)(Y,Z)- 3\,(\nabla^g_{fZ} F)(Y,fX)
 +(\nabla^g_Z F)(X,Y) -(\nabla^g_{fX} F)(Z,fY) .
\end{align*}
Adding the equation above to \eqref{3fcond}, gives 
\begin{equation}\label{fnew}
 (\nabla^g_{fX} F)(fY,Z)=-(\nabla^g_{X} F)(Y,Z).
\end{equation}
Substituting (\ref{fnew}) into (\ref{3fcond}), we obtain 
$(\nabla^g_X F)(Y,Z)=(\nabla^g_Z F)(X,Y)$; hence, in view of \eqref{E-3.3}, the condition \eqref{codazzi} is true. 
Finally, from (\ref{Eq-phi-T0}) and (\ref{fnew}) we get 
$T(Y,Z,X) =-(\nabla^g_{X} F)(Y,Z)$ that, in view of \eqref{codazzi}, completes the proof of~\eqref{Eq-phi-T0-skew}.
\end{proof}

An almost contact metric manifold $(M^{2n+1},f,\xi,\eta,g)$
is said to be \emph{almost-nearly cosymplectic},
see \cite{IZ1}, if the Levi--Civita covariant derivative of $f$ satisfies
\begin{equation*}
g\bigl((\nabla^g_X f)Y,Z\bigr) =
\tfrac13\, dF(f X,f Y,Z)
+\tfrac16\,\eta(Z)\,d\eta(Y,f X)
-\tfrac12\,\eta(Y)\,d\eta(f Z,X) ,
\end{equation*}
where $F(X,Y)=g(X,fY)$, in particular, $\xi$ is a Killing vector field.

\begin{corollary}[see \cite{IZ1}]
Let  $(M^{2n+1},f,\xi,\eta,g)$  be an almost contact metric manifold
with fundamental $2$--form $F$, considered as 
a nonsymmetric
Riemannian manifold 
$(M,G=g+F)$.
Then an Einstein connection $\nabla$ on $M$ has a totally skew-symmetric torsion $T$ if and only if the manifold is almost--nearly cosymplectic, 
in this case,
$T(X,Y,Z) =-\tfrac13\, dF(X,Y,Z)$.
\end{corollary}

The next results
explicitly present the special  Einstein connection 
of $(M^{2n+1},f,\xi,\eta,g)$
satisfying the $f^2$-torsion condition.

\begin{theorem}\label{thm72-s}
Let $\nabla$ be a special Einstein connection satisfying the $f^2$-torsion condition \eqref{E-Q-torsion}
on an almost contact metric manifold $(M^{2n+1},f,\xi,\eta,g)$, considered as an NGT space $(M,G=g+F,\nabla)$, where $F(X,Y)=g(X,f Y)$. Then 
\eqref{eq:TXiZX} 
is true,
and the torsion $T$ of $\nabla$ 
is given by  
\begin{equation}\label{spect}
 2\,T(X,Y,Z) = -(\nabla^g_XF)(Y,Z) 
 + (\nabla^g_{fY}F)(fX,Z) 
 +(\nabla^g_{fZ}F)(fX,Y) .
\end{equation}
\end{theorem}


\begin{proof}

Let $X,Y,Z$ be horizontal vector fields ($X,Y,Z\in\ker\eta$). Theorem~\ref{thm72} gives
\begin{equation}\label{eqA}
2(\nabla^g_XF)(Y,Z)=T(Y,fZ,fX)+T(fY,Z,fX).
\end{equation}
Using \eqref{E-cond-2K} (since \(\nabla\) is a special Einstein connection), and \(f^2=-I\) on 
\(\ker\eta\), we~obtain
\begin{align*}
 T(Y,fZ,fX) & =-T(X,fZ,fY)=T(fZ,X,fY)=-T(X,Y,Z), \\
 T(fY,Z,fX) & =-T(X,Z,f^2Y)=T(X,Z,Y).
\end{align*}
Hence \eqref{eqA} becomes
\begin{equation}\label{eqB}
2(\nabla^g_XF)(Y,Z)=-T(X,Y,Z)+T(X,Z,Y).
\end{equation}
Next, replacing \(X\) by \(fZ\), \(Y\) by \(fX\), 
and \(Z\) by \(Y\) in \eqref{eqA}, we get
\begin{align}\label{eqC}
\notag
 2(\nabla^g_{fZ}F)(fX,Y) &=T(fX,fY,f^2Z)+T(f^2X,Y,f^2Z) \\
                         &=-T(fX,fY,Z)+T(X,Y,Z).
\end{align}
Similarly, replacing \(X\) by \(fY\), \(Y\) by \(fX\), 
and \(Z\) by \(Z\) in \eqref{eqA}, we obtain
\begin{equation}\label{eqD}
2(\nabla^g_{fY}F)(fX,Z)=-T(fX,fZ,Y)+T(X,Z,Y).
\end{equation}
Also, applying \eqref{E-cond-2K}, we get
\begin{equation}\label{eqE}
T(fX,fY,Z)+T(fX,fZ,Y)= -T(fY,fX,Z)-T(fZ,fX,Y) =0.
\end{equation}
Subtracting \eqref{eqC} and \eqref{eqD} from \eqref{eqB}, and using \eqref{eqE}, we obtain
\[
\begin{aligned}
&2(\nabla^g_XF)(Y,Z)
-2(\nabla^g_{fZ}F)(fX,Y)
-2(\nabla^g_{fY}F)(fX,Z) \\
&=
\bigl(-T(X,Y,Z)+T(X,Z,Y)\bigr)
-\bigl(-T(fX,fY,Z)+T(X,Y,Z)\bigr) \\
&\qquad
-\bigl(-T(fX,fZ,Y)+T(X,Z,Y)\bigr) \\
&= -2\,T(X,Y,Z)
+T(fX,fY,Z)+T(fX,fZ,Y) 
= -2\,T(X,Y,Z).
\end{aligned}
\]
Thus \eqref{spect} is true
for all horizontal vector fields \(X,Y,Z\). 
If one of the vector fields is $\xi$, both sides vanish, hence the identity \eqref{spect}  holds for arbitrary vector fields on $M$. 
\end{proof}

\begin{corollary}
Under the conditions of Theorem~\ref{thm72-s},
the torsion $T$ 
(of a special Einstein connection $\nabla$)
is totally skew-symmetric if and only if
\eqref{eq:TXiZX} 
is true;
in this case,
\begin{equation}\label{spect2}
T(X,Y,Z)={\color{red}-}\tfrac12(\nabla^g_X F)(Y,Z) ={\color{red}-}\tfrac 16\,dF(X,Y,Z).
\end{equation}
\end{corollary}

\begin{proof}
Assume that the torsion $T$ is totally skew-symmetric.
Then, by Theorem~\ref{thm72-s}, we have \eqref{eq:TXiZX}
and
\eqref{spect}. Since $T$ is totally skew-symmetric,  
using \eqref{spect}, we get
\begin{equation}\label{nnnn}
(\nabla^g_{fZ}F)(fX,Y)+(\nabla^g_{fY}F)(fX,Z)=0.
\end{equation}
Using this in \eqref{spect}, we obtain
the first equality in \eqref{spect2}:
$T(X,Y,Z)=-\tfrac12(\nabla^g_XF)(Y,Z)$.

Replacing $Y$ by $fY$ and $Z$ by $fZ$ in \eqref{nnnn},
and using $f^2=-I$ on $\ker\eta$, together with
\[
(\nabla^g_XF)(fY,fZ)=-(\nabla^g_XF)(Y,Z),
\]
we obtain
\[
(\nabla^g_XF)(Y,Z)=(\nabla^g_YF)(Z,X)=(\nabla^g_ZF)(X,Y).
\]
Therefore,
\[
dF(X,Y,Z)
=(\nabla^g_XF)(Y,Z)+(\nabla^g_YF)(Z,X)+(\nabla^g_ZF)(X,Y)
=3(\nabla^g_XF)(Y,Z);
\]
hence 
\[
T(X,Y,Z)=-\tfrac12(\nabla^g_XF)(Y,Z)
=-\tfrac16\,dF(X,Y,Z).
\]
that completes the proof of \eqref{spect2}.
Conversely, assume that
\[
T(X,Y,Z)=-\tfrac12(\nabla^g_XF)(Y,Z).
\]
Since $(\nabla^g_XF)(Y,Z)$ is skew-symmetric in $Y$ and $Z$, it follows
 that $T$ is skew-symmetric in $Y$ and $Z$; therefore, $T$ is totally skew-symmetric.
\end{proof}

\section{Einstein connection of an NGT space}
\label{sec:w-Herm}

In this section, we  explicitly present the Einstein connection of a nonsymmetric pseudo-Riemannian space $(M,G=g+F)$ with non-degenerate $F$, in particular, of a weak almost Hermitian manifold.
Recall that a pseudo-Riemannian manifold $(M^{2n},g)$ endowed with a skew-symmetric (1,1)-tensor $f$ of maxi\-mal rank $2n$ ($f^2$ is not necessarily equal to $-{I}$) is called a \emph{weak almost Hermitian manifold} \cite{rov-survey24}.
In this case, we introduce a new (1,1)-tensor $\widetilde Q=-f^2-{I}$, and the $f^2$-\textit{torsion condition} \eqref{E-Q-torsion} in terms of 
$\widetilde Q$~reads
\begin{align}\label{E-tildeQ-torsion}
 T(\widetilde QX, Y,Z) = T(X,\widetilde QY,Z) = T(X,Y,\widetilde QZ) .
\end{align} 

The following lemma is used 
in the proof of Theorem~\ref{T-8.2}.

\begin{lemma}
Let $(M^{2n},f,g)$ be a weak almost Hermitian manifold, considered as a nonsymmetric pseudo-Riemannian manifold $(M,G=g+F)$, where $F(X,Y)=g(X,fY)$.
If~an Einstein connection $\nabla$ on $M$ satisfies the $f^2$-torsion condition \eqref{E-Q-torsion},
then
\begin{align}
\label{E-fy-fz2}
\notag
T(fY,fZ,({I}{-}f^2)X) =&\, T(Y,Z,({I}{-}f^2)f^2X)
{-} 2(\nabla^g_{f^2X}\,F)(Y,fZ) {+}2(\nabla^g_{f^2X}\,F)(fY,Z) \\
& +dF(Y,Z,\, ({I}-f^2)f^2X) -dF(fY,fZ,\, ({I}-f^2)X) , \\
\label{E-fy-fz3}
\notag
 T(fY,Z,\, ({I}-f^2)fX) =&\, -T(Y,fZ,\, ({I}-f^2)fX)
 -T(Y,Z,\, ({I}-f^4)X) \\
\notag 
 &+dF(Y,Z,\, ({I}-f^4)X)
  +2\,(\nabla^g_{({I}-f^2)X}\,F)(Y,Z) \\
 &-2\,(\nabla^g_{f^2X}\,F)(fY,Z) +2\,(\nabla^g_{f^2X}\,F)(Y,fZ) .
\end{align} 
\end{lemma}

\begin{proof}
Using $Z\to fZ$ in \eqref{2.7^*}
and $\widetilde Q=-f^2-{I}$, we get
\begin{align*}
 2(\nabla^g_{X}\,F)(Y,fZ)
 = &-T(fZ,X,Y) -T(X,Y,fZ) +T(Z,X,fY) +T(\widetilde QZ,X,fY) \\
 & +T(X,fY,Z) +T(X,fY,\widetilde QZ) 
 -T(Y,Z,fX) -T(Y,\widetilde QZ,fX) \\
 & +T(fY,fZ,fX) . 
\end{align*} 
Using $Y\to fY$ in \eqref{2.7^*}, we get
\begin{align*}
 2(\nabla^g_{X}\,F)(fY,Z)
 =& -T(Z,X,fY) -T(X,fY,Z) +T(fZ,X,Y) +T(fZ,X,\widetilde QY) \\
 & +T(X,Y,fZ) +T(X,\widetilde QY,fZ) 
 -T(Y,Z,fX)  -T(\widetilde QY,Z,fX) \\
 & +T(fY,fZ,fX) . 
\end{align*} 
Subtracting the above equations and using 
the $f^2$-torsion condition \eqref{E-Q-torsion}, we get
\begin{align*} 
 & 2\{T(Z,X,fY) +T(X,fY,Z)\}
 +\{T(Z,\widetilde QX,fY) +T(\widetilde QX,fY,Z)\} \\
 &-2\{T(X,Y,fZ) +T(fZ,X,Y)\}
 -\{T(\widetilde QX,Y,fZ) +T(fZ,\widetilde QX,Y)\} \\
 & =  2(\nabla^g_{X}\,F)(Y,fZ) - 2(\nabla^g_{X}\,F)(fY,Z). 
\end{align*} 
Applying \eqref{ein8} to the above equation and using the $f^2$-torsion condition \eqref{E-Q-torsion}, we get 
\begin{align}\label{E-fy-fz1}
\notag
 T(fY,Z,\, ({I}-f^2)X) =&\, T(Y,fZ,\, ({I}-f^2)X)
 - 2(\nabla^g_{X}\,F)(Y,fZ) + 2(\nabla^g_{X}\,F)(fY,Z) \\
 & +dF(Y,fZ,\, ({I}-f^2)X) -dF(fY,Z,\, ({I}-f^2)X) .
\end{align} 
Applying $Z\to fZ$ to \eqref{E-fy-fz1} gives~\eqref{E-fy-fz2}.
Using \eqref{ein8} and \eqref{E-fy-fz2}, we rewrite \eqref{2.7^*} as~\eqref{E-fy-fz3}.
\end{proof}

\begin{remark}\rm
For $f=J$, \eqref{E-fy-fz2} and \eqref{E-fy-fz3} reduce to the following:
\begin{align*}
T(JY,JZ,X) =&\, -T(Y,Z,X)
 -dF(Y,Z,X) -dF(JY,JZ,X) , \\
 T(JY,Z,JX) =&\, -T(Y,JZ,JX)
 {+}2\,(\nabla^g_{X}\,F)(Y,Z) 
 {+}(\nabla^g_{X}\,F)(JY,Z) {-}(\nabla^g_{X}\,F)(Y,JZ) .
\end{align*} 
From \eqref{E-fy-fz2} with $f=J$ we get 
\begin{align*}
 T(JY,JZ,X) =&\, -T(Y,Z,X) -dF(Y,Z,X) -dF(JY,JZ,X) \\
 &+ (\nabla^g_{X}\,F)(Y,JZ) -(\nabla^g_{X}\,F)(JY,Z) .
\end{align*} 
Subtracting the above equation from \eqref{Eq-2.12}, we get a compact presentation of \eqref{Eq-2.13}:
\begin{align*}
2\,T(Y,Z,X) = (\nabla^g_{2\,JX-X}\,F)(JY,Z) 
+ (\nabla^g_{X}\,F)(Y,JZ) 
-dF(Y,Z,X) -dF(JY,JZ,X)  .
\end{align*} 
\end{remark}

The following theorem generalizes  \cite[Theorem 1]{Prvanovic-95},
and for a non-degenerate tensor $P={I}-f^2$ we can use it to completely determine~$T$, and hence, an Einstein connection $\nabla$.

\begin{theorem}\label{T-8.2}
Let $(M^{2n},f,g)$ be a weak almost Hermitian manifold, considered as a nonsymmetric pseudo-Riemannian manifold $(M,G=g+F)$, where $F(X,Y)=g(X,fY)$.
Suppose that an Einstein connection $\nabla$ on $M$ satisfies the $f^2$-torsion condition \eqref{E-Q-torsion},
and the tensor $P:={I}-f^2$ has rank $2n$. 
Then the torsion $T$ of $\nabla$ is given~by
\begin{align}\label{Eq-QT-solution}
\notag
 & 2\,T(Y,Z,\,(I+\tfrac12\,\widetilde Q)^2f^4 X) = 
 (\nabla^g_X F)((f+f^3)Y,fZ) 
 +(\nabla^g_Y F)(f^2Z,X) 
 +(\nabla^g_Z F)(f^2X,Y) \\
\notag
&\qquad +(\nabla^g_{fX} F)(f^3Y,Z) 
-(\nabla^g_{fX} F)(f^2Y,fZ) 
+(\nabla^g_{fY} F)(f^3Z,X)
+(\nabla^g_{fZ} F)(f^2X,fY) \\
\notag 
&\qquad+(\nabla^g F)(P^{-1}(2\,\widetilde Q+\widetilde Q^2)f^2X,fY,Z)
-(\nabla^g F)(P^{-1}(2\,\widetilde Q+\widetilde Q^2)f^2X,Y,fZ) \\
&\qquad-(\nabla^g F)(\widetilde Qf^2X,Y,Z) 
-\tfrac12\,dF((3\,\widetilde Q+2\,\widetilde Q^2)X,fY,fZ) 
-\tfrac32\,dF(\widetilde Q f^4 X,Y,Z) .
\end{align}
\end{theorem}

\begin{proof} 
Applying $Y\to fY$ and $Z\to fZ$ to \eqref{2.7^*}, we have
\begin{align}\label{2.7^{***}}
\notag
& 2(\nabla^g_X F)(fY,fZ) = -T(Z,X,Y) -T(X,Y,Z) \\
\notag
&\qquad -T(fZ, X,fY) -T(X,fY,fZ) -T(fY,Z,fX)-T(Y,fZ,fX) \\
\notag
&\qquad -T(\widetilde QZ,X,Y) -T(Z,X,\widetilde QY) 
 -T(X,\widetilde QY,Z) -T(X,Y,\widetilde QZ) \\
&\qquad -T(\widetilde QZ,X,\widetilde QY) 
-T(X,\widetilde QY,\widetilde QZ) 
-T(fY,\widetilde QZ,fX) -T(\widetilde QY,fZ,fX) . 
\end{align} 
From \eqref{2.7^*} with $Y\to -f^2Y$,
we get 
\begin{align}\label{2.7^{**}}
\notag
&-2(\nabla^g_X\,F)(f^2Y,Z) = -T(Z,X,Y) -T(X,Y,Z) 
-T(fZ,X,fY) -T(X,fY,fZ)\\
\notag
&\qquad +T(Y,fZ,fX) +T(fY,Z,fX) 
-T(Z,X,\widetilde QY) -T(X,\widetilde QY,Z)\\
&\qquad 
 -T(fZ,X,\widetilde QfY)  
-T(X,\widetilde QfY,fZ)
 +T(\widetilde QY,fZ,fX) +T(\widetilde QfY,Z,fX) . 
\end{align} 
The relation \eqref{2.7^{***}} together with \eqref{2.7^{**}} yields the following:
\ 
\begin{align}\label{2.8^*}
\notag
&-T(X,fY,fZ) -T(fZ,X,fY) = T(X,Y,Z) +T(Z,X,Y) +T(Z,X,\widetilde QY) \\
\notag
&+T(X,\widetilde QY,Z) 
+\tfrac12\big\{ 
T(fZ,X,\widetilde QfY)+T(X,\widetilde QfY,fZ) 
 +T(\widetilde QZ,X,Y)  \\
\notag 
&+T(\widetilde QZ,X,\widetilde QY) +T(X,Y,\widetilde QZ)
+T(X, \widetilde QY,\widetilde QZ) -T(\widetilde QfY,Z,fX) \\
&+T(fY,\widetilde QZ,fX) \big\}
 +(\nabla^g_X F)(fY,fZ) -(\nabla^g_X\,F)(f^2Y,Z) .
\end{align} 
Using \eqref{2.8^*}, we reduce \eqref{2.7^{**}} to 
the following: 
\begin{align}\label{2.9^*}
\notag
&-(\nabla^g_X\,F)(f^2Y,Z) {-}(\nabla^g_X\,F)(fY,fZ)
= T(Y,fZ,fX) {+}T(fY,Z,fX) {+}T(\widetilde QY,fZ,fX) \\
\notag
&\quad + \tfrac12\big\{T(\widetilde QZ,X,\widetilde QY)
-T(fZ,X, \widetilde Q fY) -T(X, \widetilde QfY,fZ)  
+T(\widetilde QZ,X,Y)  \\
&\quad +T(X,Y,\widetilde QZ) +T(X,\widetilde QY,\widetilde QZ) 
+T(\widetilde QfY,Z, fX) 
+T(fY,\widetilde Q Z,fX) \big\}.
\end{align} 
Taking the cyclic sum of \eqref{2.9^*}, 
in view of \eqref{2.8^*}, we get: 
\begin{align}\label{2.10^*}
\notag
&
-(\nabla^g_X\,F)(f^2Y,Z) 
-(\nabla^g_Y\,F)(f^2Z,X) -(\nabla^g_Z\,F)(f^2X,Y) \\
&\qquad = -T(Y,Z,X) -T(X,Y,Z) -T(Z,X,Y) +\tfrac12\,\delta_1(X,Y,Z) ,
\end{align} 
where $\delta_1(X,Y,Z)$ is given in Section~\ref{sec:app}.

From \eqref{2.10^*} with $Y\to fY$ and $Z\to fZ$,
using \eqref{2.8^*}, we obtain 
\begin{align}\label{2.11^*}
\notag
&-(\nabla^g_X\,F)(f^3Y,fZ) -(\nabla^g_{fY}\,F)(f^3Z,X) 
-(\nabla^g_{fZ}\,F)(f^2X,fY)\\
\notag
& +(\nabla^g_X\,F)(f^2Y,Z) -(\nabla^g_{X}\,F)(fY,fZ) \\
&\quad = T(X,Y,Z) +T(Z,X,Y) -T(fY,fZ,X) +\tfrac12\,\delta_2(X,Y,Z), 
\end{align} 
where $\delta_2(X,Y,Z)$ is given in Section~\ref{sec:app}.
Replacing $X\to fX$ and $Y\to fY$ in \eqref{2.9^*}, 
we get: 
\begin{align}\label{2.12^*}
\notag
& T(fY,fZ,X) = T(Y,Z,X) +(\nabla^g_{fX}\,F)(f^2Y,fZ) +(\nabla^g_{fX}\,F)(f^3Y,Z) 
\\
\notag
&+\tfrac12\,T(fZ,fX,\widetilde QY) +\tfrac12\,T(fZ,fX,\widetilde Q^2Y) +\tfrac12\,T(fX,\widetilde QY,fZ) +\tfrac12\,T(fX,\widetilde Q^2Y,fZ) \\ 
\notag
&+\tfrac32\,T(\widetilde QY,Z,X) +\tfrac12\,T(\widetilde Q^2Y,Z,X) +\tfrac32\,T(\widetilde QY,Z,\widetilde QX) +\tfrac12\,T(\widetilde Q^2Y,Z,\widetilde QX) \\
\notag
&+\tfrac12\,T(\widetilde QZ,fX,fY) +\tfrac12\,T(\widetilde QZ,fX,\widetilde QfY) +\tfrac12\,T(fX,fY,\widetilde QZ) +\tfrac12\,T(fX,\widetilde QfY,\widetilde QZ) \\ 
\notag
&+\tfrac12\,T(Y,\widetilde QZ,X) +\tfrac12\,T(\widetilde QY,\widetilde QZ,X) +\tfrac12\,T(Y,\widetilde QZ,\widetilde QX) +\tfrac12\,T(\widetilde QY,\widetilde QZ,\widetilde QX) \\
&-T(fY,fZ,\widetilde QX) -T(\widetilde QfY,fZ,X) -T(\widetilde QfY,fZ,\widetilde QX) +T(Y,Z,\widetilde QX) . 
\end{align} 
Then substituting \eqref{2.12^*} into \eqref{2.11^*}, 
we get the following:  
\begin{align*}
& (\nabla^g_X F)(f^2Y, Z) -(\nabla^g_X F)((f+f^3)Y, fZ) 
-(\nabla^g_{fY} F) (f^3Z, X)  \\
&-(\nabla^g_{fZ} F)(f^2X, fY) 
+(\nabla^g_{fX} F)(f^3Y, Z) +(\nabla^g_{fX} F)(f^2Y, fZ) \\
&\quad = 
T(X,Y,Z) +T(Z,X,Y) -T(Y,Z,X) +\tfrac12\,\delta_3(X,Y,Z) ,
\end{align*} 
where $\delta_3(X,Y,Z)$ is given in Section~\ref{sec:app}.
This relation together with \eqref{2.10^*} yields 
 \begin{align}\label{E-Q-fy-fz}
\notag
 2\,T(Y,Z,X) =&\ (\nabla^g_X F)((f+f^3)Y,fZ) +(\nabla^g_Y F)(f^2Z,X) 
+(\nabla^g_Z F)(f^2X,Y) \\
\notag
&-(\nabla^g_{fX} F)(f^3Y,Z) -(\nabla^g_{fX} F)(f^2Y,fZ)
+(\nabla^g_{fY} F)(f^3Z,X) 
\\
& +(\nabla^g_{fZ} F)(f^2X,fY) +\tfrac12\,\delta_4(X,Y,Z) ,
\end{align}
where $\delta_4(X,Y,Z)$ is given in Section~\ref{sec:app}.
Using \eqref{ein8},
$-f^2{=}I+\widetilde Q$,
and the $f^2$-torsion condition \eqref{E-tildeQ-torsion}, we reduce $\delta_4(X,Y,Z)$ to $\delta_5(X,Y,Z)$, given in Section~\ref{sec:app}. 
Using \eqref{E-fy-fz2} and \eqref{E-fy-fz3}, we~get
\begin{align*}
 & T(fY,fZ,(3\,\widetilde Q +2\,\widetilde Q^2)X) 
 = T(Y,Z, (3\,\widetilde Q+2\,\widetilde Q^2)f^2X) 
 +dF((3\,\widetilde Q+2\,\widetilde Q^2)f^2X,Y,Z) \\
 &\qquad -dF((3\,\widetilde Q+2\,\widetilde Q^2)X,fY,fZ)
 -2\,(\nabla^g F)(P^{-1}(3\,\widetilde Q+2\,\widetilde Q^2)f^2X,Y,fZ) \\
 &\qquad +2\,(\nabla^g F)(P^{-1}(3\,\widetilde Q+2\,\widetilde Q^2)f^2X,fY,Z), \\
 & T(Y,fZ,\widetilde Qf^3X) +T(fY,Z,\widetilde Qf^3X) 
= T(Y,Z, \widetilde Q^2 f^2 X) 
-dF(\widetilde Q^2 f^2 X,Y,Z) \\
&\qquad 
 +2\,(\nabla^g F)(\widetilde Qf^2X,Y,Z) 
 +2\,(\nabla^g F)(P^{-1}\widetilde Qf^4X,Y,fZ) 
 -2\,(\nabla^g F)(P^{-1}\widetilde Qf^4X,fY,Z) , 
\end{align*}
which allow us to transform some terms in
$\delta_5(X,Y,Z)$ and then rewrite it as  
\begin{align*}
\notag
\delta_6(X,Y,Z) =&\,
-T(Y,Z,\ (12\,\widetilde Q+13\,\widetilde Q^2 +6\,\widetilde Q^3+\widetilde Q^4)X)
-3\,dF(\widetilde Q f^4X,Y,Z) \\
&-dF((3\,\widetilde Q+2\,\widetilde Q^2)X,fY,fZ) 
-2\,(\nabla^g F)(P^{-1}(2\,\widetilde Q+\widetilde Q^2)f^2X,Y,fZ) \\ 
 &+2\,(\nabla^g F)(P^{-1}(2\,\widetilde Q+\widetilde Q^2)f^2X,fY,Z) 
 -2\,(\nabla^g F)(\widetilde Qf^2X,Y,Z) .
\end{align*}
Note that $\delta_i(X,Y,Z)\ (1\le i\le 6)$ vanish when $\widetilde Q=0$.
From the equation \eqref{E-Q-fy-fz} with $\delta_4(X,Y,Z)$ replaced by $\delta_6(X,Y,Z)$, the equation \eqref{Eq-QT-solution} follows.
Using \eqref{E-3.3}, we can further simplify \eqref{Eq-QT-solution} to the form with $\nabla^g\,F$-terms and without $dF$-terms.
\end{proof}

\begin{example}\rm
Take 
almost Hermitian 
manifolds $(M_j, J_j,g_j)\ (j=1,\dots,m)$, where $J_j^2=-{I}_{\,j}$.
Consider the Riemannian product
$(M,g)=\bigl(M_1\times\cdots\times M_m,\; g_1\oplus\cdots\oplus g_m\bigr)$,
and define the $(1,1)$--tensor 
$f=\sqrt{\lambda_1}\,J_1\oplus\ldots\oplus\sqrt{\lambda_m}\,J_m$, where $\lambda_j\in\mathbb{R}_+$.
Then $(M,f,g)$ is a weak almost Hermitian manifold
with $f^2=-\bigoplus_{j}\lambda_j\,{I}_{j}$, and on $TM_j$ and for $X\in TM_j$ we~get
\begin{align*}
 fX=\sqrt{\lambda_j}\,J_jX,\quad 
f^2X=-\lambda_jX,\quad f^3X=-\lambda_j^{3/2}J_jX,\quad 
f^4X=\lambda_j^2X,
\quad f^6X = -\lambda_j^{3} X , \\
 \widetilde Q|_{\,TM_j}=(\lambda_j-1)\,{I}_j, \quad
P={I}-f^2=\bigoplus\nolimits_j(1+\lambda_j)\,{I}_j,
\quad
P^{-1}=\bigoplus\nolimits_j\frac{1}{1+\lambda_j}\,{I}_j.
\end{align*} 
Since the Levi-Civita connection of the Riemannian product preserves the splitting $TM=\bigoplus_{j=1}^m TM_j$, $f^2$ acts as a 
multiple of the identity on each factor.
The 2-form $F$ corresponding to $f$ is given by 
$F=F_1\oplus\ldots\oplus F_m$, 
where $F_j(X,Y)=g_j(X,\sqrt\lambda_j J_jY)$. As a consequence,
\[
 (\nabla^g_XF)(Y,Z)=0,
\]
whenever $X$ belongs to one factor and at least one of $Y,Z$ belongs to a
different factor.

If an Einstein connection $\nabla$ satisfies the $f^2$-torsion condition \eqref{E-Q-torsion}, then  
 $T(X,Y)=0$ for $X\in TM_i,\, Y\in TM_j\ (i\ne j)$.
Hence, the torsion tensor splits into the direct sum
 $T=T|_{TM_1}\oplus\ldots\oplus T|_{TM_m}$.
Substituting $\widetilde Q|_{\,TM_j} = (\lambda_j-1)\,{I}_j$ into \eqref{Eq-QT-solution}, then using
\[
 (\nabla^{g_j}_X F_j)(J_j Y, J_j Z)= -\,(\nabla^{g_j}_X F_j)(Y,Z),
 \quad
 (\nabla^{g_j}_X F_j)(J_j Y, Z)= (\nabla^{g_j}_X F_j)(Y, J_j Z),
\] 
we uniquely determine the torsion component $T|_{TM_j}$
by
\begin{align*}
T|_{\,TM_j}(Y,Z,X)
&=\lambda_j^{-2}(\lambda_j-1)(\nabla^{g_j}_X F_j)(Y,Z)
-\tfrac12\lambda_j^{-2}(\nabla^{g_j}_Y F_j)(Z,X)
-\tfrac12\lambda_j^{-2}(\nabla^{g_j}_Z F_j)(X,Y)\\
&\ 
+\lambda_j^{-1}(\nabla^{g_j}_{J_jX} F_j)(Y,J_jZ)
{-}\tfrac12\lambda_j^{-1}(\nabla^{g_j}_{J_jY} F_j)(J_jZ,X)
{-}\tfrac12\lambda_j^{-1}(\nabla^{g_j}_{J_jZ} F_j)(X,J_jY)\\
&\ 
-\tfrac14\,\lambda_j^{-2}(\lambda_j-1)(3\lambda_j+1)dF_j(X,Y,Z)
+\tfrac14(\lambda_j-1)\lambda_j^{-1/2}dF_j(J_jX,Y,J_jZ)\\
&\
+\tfrac14(\lambda_j-1)\lambda_j^{-1/2}dF_j(J_jX,J_jY,Z)
-\tfrac12\lambda_j^{-2}(\lambda_j-1)dF_j(X,J_jY,J_jZ) .
\end{align*}
If each $(M_j,J_j,g_j)$ is a K\"ahler manifold, then $\nabla^{g_j}F_j=0$ and $dF_j=0$. In this case, all torsion components vanish, and the Einstein connection reduces to the Levi-Civita connection.
Using \eqref{E-3.3}, we can further simplify the formula for $T|_{TM_j}\ (1\le j\le m)$ as
\begin{align*}
T|_{TM_j}(Y,Z,X) = &\ \frac{(\lambda_j-1)
(5-3\lambda_j)}{4\lambda_j^{2}}\,
(\nabla^{g_j}_X F_j)(Y,Z)
+\Bigl(\frac{1}{\lambda_j}+\frac{\lambda_j-1}{2\sqrt{\lambda_j}}\Bigr)\,
(\nabla^{g_j}_{J_jX} F_j)(Y,J_jZ) \\
&-\Bigl(\frac{3\lambda_j^{2}-2\lambda_j+1}{4\lambda_j^{2}}
+\frac{\lambda_j-1}{4\sqrt{\lambda_j}}\Bigr)\,
\big\{(\nabla^{g_j}_Z F_j)(X,Y)
+(\nabla^{g_j}_Y F_j)(Z,X) \big\}\\
&+\Bigl(
\frac{\lambda_j-1}{4\sqrt{\lambda_j}} 
-\frac{2\,\lambda_j-1}{2\lambda_j^{2}}\Bigr)\,
\big\{
 (\nabla^{g_j}_{J_jY} F_j)(J_jZ,X)
+(\nabla^{g_j}_{J_jZ} F_j)(X,J_jY)
\big\}.
\end{align*}
The formula above for $\lambda_j=1$ gives the solution \eqref{Eq-2.13}. 
\end{example}

For a nonsymmetric Riemannian space $(M,G=g+F)$,
the tensor $P={I}-f^2$ is positive definite, and hence, non-degenerate. 
Therefore, we have the following.

\begin{corollary}\label{C-5.2}
Let $\nabla$ be an Einstein connection of a weak almost Hermitian manifold $(M^{2n},f,g>0)$, considered as a nonsymmetric Riemannian space 
with $F(X,Y)=g(X,fY)$.  
If~the 
condition \eqref{E-Q-torsion} is valid,  
then the torsion $T$ of $\nabla$ is given by~\eqref{Eq-QT-solution}.
\end{corollary}



\begin{corollary}[see Theorems 1 and 2 of \cite{Prvanovic-95}]
Let $(M^{2n},J,g)$ be an almost Hermitian manifold considered as a nonsymmetric Riemannian manifold $(M,G=g+F)$, where $F(X,Y)=g(X,JY)$. 
Then the torsion of an Einstein connection $\nabla$ on $M$ is given by
\begin{align}\label{Eq-2.13}
\notag
 2\,T(Y,Z,X) 
 =&\ 2\,(\nabla^g_{JX} F)(JY,Z)
  - (\nabla^g_{JY} F)(JZ,X) -(\nabla^g_{JZ} F)(X,JY) \\
 &-(\nabla^g_Y F)(Z,X) -(\nabla^g_Z F)(X,Y) ;
\end{align}
in particular, the torsion of a special Einstein connection $\nabla$ on $M$, see \eqref{E-cond-2K}, is given~by  
\begin{equation*}
 2\,T(X,Y,Z) = (\nabla^g_XF)(Y,Z) 
 -(\nabla^g_{JZ}F)(JX,Y) -(\nabla^g_{JY}F)(JX,Z) .
\end{equation*}
\end{corollary}


\begin{remark}
\rm
In view of \eqref{Eq-2.13}, the Einstein connection $\nabla$ on an almost Hermitian manifold is special, see definition (\ref{E-cond-2K}), if and only if the following identity holds:
\begin{equation}\label{s1}
(\nabla^g_X F)(Y,Z)+(\nabla^g_Y F)(X,Z)
=(\nabla^g_{JX}F)(JY,Z)+(\nabla^g_{JY}F)(JX,Z).
\end{equation} 
Among the sixteen Gray-Hervella classes \cite{gh-1980} of almost Hermitian manifolds,
the condition (\ref{s1}) of a special Einstein connection is satisfied for the following classes:
\[
{\mathcal W}_1,\ \ 
{\mathcal W}_3,\ \ 
{\mathcal W}_4,\ \
{\mathcal W}_3\oplus {\mathcal W}_4,\ \
{\mathcal W}_1\oplus {\mathcal W}_3,\ \
{\mathcal W}_1\oplus {\mathcal W}_4,\ \
{\mathcal W}_1\oplus {\mathcal W}_3\oplus {\mathcal W}_4;
\]
consequently, the Einstein connection on an almost Hermitian manifold
belonging to any of these classes is a special Einstein connection, see Theorem 3 of \cite{Prvanovic-95}.
\end{remark}


\begin{example}
\rm
If~$(\nabla^g_X J)X=0\ (X\in \mathfrak{X}_M)$, 
or $\nabla^g J=0$,
then a weak almost Hermitian manifold $(M^{2n}, J, g)$ is 
a 
\textit{weak nearly K\"{a}hler~manifold},
or a \textit{weak K\"{a}hler~manifold}, respectively,~\cite{rov-survey24}.
Let $(M, f, g)$ be a \textit{weak nearly K\"{a}hler manifold}, considered as a nonsymmetric Riemannian manifold $(M, G = g + F)$ with $F(X,Y)=g(X,fY)$.
If $\nabla$ is an Einstein connection on $M$ with totally skew-symmetric torsion $T$, then the 
$f$-\textit{torsion condition}~holds:
\begin{align}\label{E-f-torsion} 
 T(fX, Y) = T(X,fY) = -f\,T(X,Y),
\end{align}
and the torsion of $\nabla$ is determined by
 $T(X,Y,Z)=-\tfrac13\,dF(X,Y,Z)$,
see Example 3.4 in~\cite{rst-63}. 
In~this case, \eqref{newnbl-K} reduces to
$K_XY=\tfrac12\,T(X,Y)$, and 
 $\nabla_X Y=\nabla^g_X Y + \tfrac12\,T(X,Y)$ holds.
\end{example}

\section{The terms $\delta_i(X,Y,Z)$}
\label{sec:app}

Here, we present $\delta_i(X,Y,Z)\ (1\le i\le 5)$, see Section~\ref{sec:w-Herm}, verified with the help of MAPLE.
\begin{align*}
&\delta_1(X,Y,Z)
= -T(fZ,X,\widetilde QfY) -T(X,\widetilde QY,Z) 
-T(X,\widetilde QfY,fZ) +T(\widetilde QfY,Z,fX)\\
& -T(Z,X,\widetilde QY) -T(X,Y,\widetilde QZ) 
+ T(\widetilde QY,fZ,fX) - T(fX,Y,\widetilde QfZ)
-T(Y,\widetilde QZ,X) \\
& -T(Y,\widetilde QfZ,fX) - T(Y,Z,\widetilde QX) + T(\widetilde QfZ,X, fY) 
- T(Z,\widetilde QX,Y) + T(\widetilde QfX,Y,fZ) \\
& - T(fY,Z,\widetilde QfX) - T(Z,\widetilde QfX,fY)
  + T(\widetilde QZ,fX,fY) + T(\widetilde QX,fY,fZ) .
\end{align*} 
\begin{align*} 
 \delta_2(X,Y,Z) =
 &\ T(X,\widetilde QY,Z) -T(\widetilde QY,fZ,fX) 
 -2\,T(\widetilde QfY,Z,fX) +T(Z,X,\widetilde QY) \\
 &+2\,T(\widetilde QZ,X,Y) {+}T(\widetilde QZ,X,\widetilde QY) {+}T(X,Y,\widetilde QZ) {+}2\,T(fY,\widetilde QZ,fX) \\
 &-T(\widetilde QfX,fY,Z) -T(\widetilde QfX,fY,\widetilde QZ) +T(fX,fY,\widetilde QZ) \\
&+T(fX,fY,\widetilde Q^2Z) +T(fY,\widetilde Q^2Z,fX) +T(Y,fZ,\widetilde QfX) \\
&+T(\widetilde QY,fZ,\widetilde QfX) -T(\widetilde Q^2Y,fZ,fX) +T(fZ,\widetilde QfX,Y) \\
&+T(fZ,\widetilde QfX,\widetilde QY) 
-T(\widetilde QfY,\widetilde QZ,fX) -T(\widetilde QfZ,fX,Y) \\
&-T(\widetilde QfZ,fX,\widetilde QY) -T(X,fY,\widetilde QfZ)
-T(fY,\widetilde QfZ,X) \\
&-T(fY,fZ,\widetilde QX) +T(\widetilde Q^2Z,X,Y) +T(\widetilde Q^2Z,X,\widetilde QY) \\
&-T(X,\widetilde Q^2Y,Z) -T(X,\widetilde Q^2Y,\widetilde QZ) 
+T(\widetilde QX,\widetilde QY,Z) \\
&+T(\widetilde QX,\widetilde QY,\widetilde QZ) -T(Z,X,\widetilde Q^2Y) -T(\widetilde QZ,X,\widetilde Q^2Y) \\
&+T(\widetilde QX,Y,Z) +T(\widetilde QX,Y,\widetilde QZ) -T(fZ,\widetilde QX,fY) ,
\end{align*} 
\begin{align*}
\delta_3(X,Y,Z)
=&\ T(fZ,\widetilde QfX,\widetilde QY)
{-}T(\widetilde QfZ,fX,\widetilde QY) {-} T(\widetilde QfZ,fX,Y) 
{-} T(\widetilde QfY,\widetilde QZ,fX) \\
&+T(fZ,\widetilde QfX,Y) {-}T(\widetilde Q^2Y,fZ,fX) 
{+}T(\widetilde QY,fZ,\widetilde QfX) {+} T(Y,fZ,\widetilde QfX) \\
&+T(fY,\widetilde Q^2Z,fX) {+} T(fX,fY,\widetilde Q^2Z) 
{-} T(\widetilde QfX,fY,\widetilde QZ) {-} T(\widetilde QfX,fY,Z) \\
&-T(\widetilde QZ,X,\widetilde Q^2Y) - T(Z,X,\widetilde Q^2Y) 
- 3\,T(\widetilde QY,Z,\widetilde QX) + T(\widetilde QX,\widetilde QY,\widetilde QZ) \\
&-T(fX,\widetilde QY,fZ) - 3\,T(\widetilde QY,Z,X) - T(fX,\widetilde QfY,\widetilde QZ) - T(\widetilde QZ,fX,\widetilde QfY) \\
&+2\,T(\widetilde QfY, fZ,\widetilde QX) {+} 2\,T(\widetilde QfY,fZ,X) {-}T(fX,\widetilde Q^2Y,fZ) {-}T(fZ,fX,\widetilde QY) \\
&-T(fZ,fX,\widetilde Q^2Y) - T(\widetilde Q^2Y,Z,X) - T(\widetilde Q^2Y,Z,\widetilde QX) - T(\widetilde QY,\widetilde QZ,X) \\
&-T(\widetilde QY,\widetilde QZ,\widetilde QX) - T(Y,\widetilde QZ,\widetilde QX) + T(\widetilde QX,\widetilde QY,Z) - T(X,\widetilde Q^2Y,\widetilde QZ) \\
&-T(X,\widetilde Q^2Y,Z) + T(\widetilde Q^2Z,X,\widetilde QY) 
+ T(\widetilde Q^2Z,X,Y) + T(fY,fZ,\widetilde QX)  \\
&- T(fY,\widetilde QfZ,X) - T(X,fY,\widetilde QfZ) - T(fZ,\widetilde QX,fY) + T(\widetilde QX,Y,\widetilde QZ) \\
&+ T(\widetilde QX,Y,Z) -T(\widetilde QY,fZ,fX) + T(X,Y,\widetilde QZ) + T(\widetilde QZ,X,\widetilde QY)  \\
&+ T(Z,X,\widetilde QY) +T(X,\widetilde QY,Z) + 2\,T(fY,\widetilde QZ,fX) + 2\,T(\widetilde QZ,X,Y) \\
&- 2\,T(\widetilde QfY,Z,fX) - T(\widetilde QZ,fX,fY) 
- 2\,T(Y,Z,\widetilde QX) - T(Y,\widetilde QZ,X) ,
\end{align*}
\begin{align*}
\delta_4(X,Y,Z) 
=&\ 2\,T(\widetilde QZ, X, Y) -T(fZ, X, \widetilde QfY) 
-T(X, \widetilde QfY, fZ) -T(\widetilde QfY, Z, fX) \\
&+T(\widetilde QZ, X, \widetilde QY) +2\,T(fY, \widetilde QZ, fX) 
-T(fX, Y, \widetilde QfZ) -2\,T(Y, \widetilde QZ, X) \\
&-T(Y, \widetilde QfZ, fX) +T(\widetilde QX, Y, Z) +T(\widetilde QX, Y, \widetilde QZ) -3\,T(Y, Z, \widetilde QX) \\
&-T(Y, \widetilde QZ, \widetilde QX) +T(\widetilde QfZ, X, fY) -T(fZ, \widetilde QX, fY) -T(Z, \widetilde QX, Y) \\
&+T(\widetilde QfX, Y, fZ) -T(fY, Z, \widetilde QfX) -T(Z, \widetilde QfX, fY) -3\,T(\widetilde QY, Z, X) \\
&-3\,T(\widetilde QY, Z, \widetilde QX) -T(fX, \widetilde QY, fZ) +T(\widetilde QX, fY, fZ) -T(X, fY, \widetilde QfZ)  \\
&-T(fY, \widetilde QfZ, X) +T(fY, fZ, \widetilde QX) +T(\widetilde Q^2Z, X, Y) +T(\widetilde Q^2Z, X, \widetilde QY) \\
&-T(X, \widetilde Q^2Y, Z)-T(X, \widetilde Q^2Y, \widetilde QZ) +T(\widetilde QX, \widetilde QY, Z) +T(\widetilde QX, \widetilde QY, \widetilde QZ) \\
&-T(Z, X, \widetilde Q^2Y) -T(\widetilde QZ, X, \widetilde Q^2Y) -T(\widetilde QfX, fY, Z) -T(\widetilde QfX, fY, \widetilde QZ) \\ &+T(fX, fY, \widetilde Q^2Z) {+}T(fY, \widetilde Q^2Z, fX) 
{+}T(Y, fZ, \widetilde QfX) {+}T(\widetilde QY, fZ, \widetilde QfX) \\
&-T(\widetilde Q^2Y, fZ, fX) {+}T(fZ, \widetilde QfX, Y) 
{+}T(fZ,\widetilde QfX, \widetilde QY) {-}T(fX,\widetilde Q^2Y, fZ) \\ 
&-T(fZ, fX, \widetilde QY) -T(fZ, fX, \widetilde Q^2Y) -T(\widetilde Q^2Y, Z, X) -T(\widetilde Q^2Y, Z, \widetilde QX) \\
&-T(\widetilde QY, \widetilde QZ, X) -T(\widetilde QY, \widetilde QZ, \widetilde QX) +2\,T(\widetilde QfY, fZ, X) \\
&+2\,T(\widetilde QfY, fZ, \widetilde QX) 
-T(\widetilde QfY, \widetilde QZ, fX) -T(\widetilde QfZ, fX, Y) \\
&-T(\widetilde QfZ, fX, \widetilde QY) 
-T(\widetilde QZ, fX, \widetilde QfY) 
-T(fX, \widetilde QfY, \widetilde QZ) , 
\end{align*}
\begin{align*}
\notag
\delta_5(X,Y,Z)  
\overset{\eqref{E-tildeQ-torsion}}=&\, 
-T(fZ, X, \widetilde QfY) +T(Z, X, \widetilde QY) +T(X, Y, \widetilde QZ) -8\,T(Y, Z, \widetilde QX) \\
&-T(fX, Y, \widetilde QfZ) -T(X, fY, \widetilde QfZ) 
+T(Z, X, \widetilde Q^2Y) -T(fX, fY, \widetilde QZ) \\
&-T(fX, fY, \widetilde Q^2Z) {+}2\,T(fY, fZ, \widetilde QX) 
{-}T(fZ, fX, \widetilde QY) {-}T(fZ, fX, \widetilde Q^2Y) \\
&+2\,T(fY, fZ, \widetilde Q^2X) {-}6\,T(Y, Z, \widetilde Q^2X) {-}2\,T(Y, Z, \widetilde Q^3X) {-}T(Z, fX, \widetilde Q^2fY) \\
&+T(X, Y, \widetilde Q^2Z) -T(fX, Y, \widetilde Q^2fZ) -T(Z, fX, \widetilde QfY) \\
\overset{\eqref{ein8}}=&\, 
-T(Y,Z, (9\,\widetilde Q+7\,\widetilde Q^2+2\,\widetilde Q^3)X) 
+T(fY,fZ,(3\,\widetilde Q+2\,\widetilde Q^2)X)\\
&-T(Y,fZ,\widetilde Qf^3X) -T(fY,Z,\widetilde Qf^3X) 
-dF(\widetilde Qf^3X,Y,fZ)\\
&-dF(\widetilde Qf^3X,fY,Z) 
 +dF(\widetilde QX,fY,fZ) 
 +dF(\widetilde Qf^2X,Y,Z) .
\end{align*}

\section*{Conclusion}

We presented explicitly the Einstein connection for nonsymmetric pseudo-Rieman\-nian, in particular, for weak almost Hermitian manifolds, satisfying the $f^2$-torsion condition.
We~also obtained the corresponding formulas for almost contact metric mani\-folds, assuming the $f^2$-torsion condition.
We expressed the torsion in terms of $\nabla^g F$ and $dF$, and the additional $\widetilde Q$-terms, explained precisely how the weak almost Hermitian case differs from the almost Hermitian case.
The~presented identities provide a new tool for constructing examples and for studying classes of nonsymmetric pseudo-Rieman\-nian manifolds, in particular, weak almost Hermitian mani\-folds
with the usual Gray-Hervella classification, in Einstein's nonsymmetric gravitational~theory.
In~further research we will study weak (para-) metric $f$-manifolds, in particular, weak (para-) contact metric manifolds, with an Einstein connections or connections with totally skew-symmetric torsion, 
satisfying the $f^2$-torsion condition.

\bigskip

\noindent{\bf Acknowledgments.} This work was 
supported for Prof. Milan Zlatanovi\' c
by the Ministry of Education, Science and Technological Development of the Republic of Serbia (contract reg. no. 451-03-34/2026-03/200124). 
The authors would also like to thank Dr. Miroslav Maksimovi\'c for carefully reading the paper and for his helpful and valuable comments.

\end{document}